\documentclass[12pt]{article}
\usepackage{amsmath,amssymb}
\usepackage[T1]{fontenc}
\usepackage{latexsym}
\usepackage{pdfsync}
\usepackage{epsfig}
\usepackage{proof}
\usepackage[utf8]{inputenc}
\usepackage[a4paper, top=3cm, bottom=3cm ]{geometry}
\newtheorem{thm}{Theorem}[section]

\newtheorem{lem}{Lemma}[section]

\title
{$L^p$-norm estimate for the Bergman projection on  Hartogs triangle}

\author{\normalsize Tomasz Beberok \\
\small Faculty of Mathematics and Computer Science, Jagiellonian University,\\
\small Lojasiewicza 6, 30-048 Krakow, Poland \\}

\date{}

\begin{document}

\begin{center}
  \textbf{$L^p$-norm estimate of the Bergman projection on the Hartogs triangle
}
\end{center}
\vskip1em
\begin{center}
  Tomasz Beberok
\end{center}

\vskip2em

\noindent \textbf{Abstract.} The purpose of this paper is to give an estimate of the $L^p$-norm of the Bergman projection on the Hartogs triangle.
\vskip1em


\section{Introduction}
In this note we show an estimate of the $L^p$-norm of the Bergman projection on the Hartogs triangle, the pseudoconvex domain in $\mathbb{C}^2$ defined as
\begin{align*}
  \mathcal{H}=\left\{(z_1,z_2) \in \mathbb{C}^{2} \colon   |z_1| < |z_2| <1  \right\},
\end{align*}
for $4/3<p<4$. The Hartogs triangle has remarkable geometric and function-theoretic properties, and is a classical source of counterexamples in complex analysis. The boundary $b\mathcal{H}$ of the domain $\mathcal{H}$ has a serious
singularity at the point 0, where $b\mathcal{H}$ cannot even be represented as a graph of a continuous function. The closure $\overline{\mathcal{H}}$ does not have a Stein neighborhood basis. Instead, it has a nontrivial Nebenh\"{u}lle. The $\overline{\partial}$-problem on $\mathcal{H}$ is not globally regular (see \cite{CC91}). \newline
Let $\mathbb{D}$ denote the unit disc in $\mathbb{C}$, and $\nu$ the normalized Lebesgue volume measure on $\mathbb{D}$, while $\sigma$ is the normalized surface measure on its boundary $\mathbb{T}$. Let $dV$ denote the Lebesgue volume measure on $\mathcal{H}$. The space $L^p_h(\mathcal{H})$ consists of all holomorphic functions $f$ on $\mathcal{H}$, for which
\begin{align*}
  \|f\|_p:=\left\{ \int_{\mathcal{H}} |f(z)|^p d\mu(z)  \right\}^{1/p} <\infty,
\end{align*}
where $d\mu= \frac{ dV}{\pi^2}$.
The orthogonal projection operator $\mathbf{P} \colon L^2(\mathcal{H}) \rightarrow L^2_h(\mathcal{H})$ is the Bergman projection associated with the domain $\mathcal{H}$. It follows from the Riesz representation theorem that the Bergman projection is an integral operator with the kernel $K_{\mathcal{H}} (z, w)$ on $\mathcal{H} \times \mathcal{H}$, i.e. $\mathbf{P} f(z) = \int_{\mathcal{H}} K_{\mathcal{H}} (z, w)f(w)\,d\mu(w)$ for all $f \in L^2_h(\mathcal{H})$ (see \cite{Kra}, section 1 for more on this topic). It is well known that (see \cite{Ed})
\begin{align*}
  K_{\mathcal{H}} ((z_1,z_2), (w_1,w_2))=\frac{z_2 \overline{w}_2}{(1-z_2 \overline{w}_2)^2 (z_2 \overline{w}_2 - z_1 \overline{w}_1)^2}.
\end{align*}
 The Bergman projection is a central object in the study of analytic function spaces. It naturally relates to fundamental questions such as duality and harmonic conjugates, and it is also a building block for Toeplitz operators. Understanding its behaviour and estimating its size is therefore of vital importance on several occasions. There are several articles on Bergman projections. We refer the reader to the following articles and the references therein \cite{Chen,Dos,KP,LS,LS2,Liu,MN2,PS,Zey} for details of this interesting topic.
 In \cite{Ch}, Chakrabarti and Zeytuncu proved that the Bergman projection $\mathbf{P}$ is a bounded operator from $L^p(\mathcal{H})$ to $L^p_h(\mathcal{H})$ if and only if $4/3<p<4$. For the interested reader, we recommend \cite{Edholm} for more general result. A natural and interesting question is to determine the exact value of the $L^p$-operator norm $\|\mathbf{P}\|_p$ of this operator. This turns out to be a difficult task to accomplish, except for the trivial case when $p =2$.
\section{Preliminaries}
\subsection{The hypergeometric function}
If the real part of the complex number $z$ is positive ($\Re (z) > 0$), then the integral
\begin{align*}
  \Gamma(z)=\int_{0}^{\infty} x^{z-1} e^{-x} \, dx
\end{align*}
converges absolutely, and is known as the Euler integral of the second kind. The recurrence relation
\begin{align}\label{gamma}
  z\Gamma(z)=\Gamma(z+1)
\end{align}
can be used to uniquely extend the integral formulation for $\Gamma(z)$ to a meromorphic function defined for all complex numbers  $z$, except integers less than or equal to zero. Other important functional equations for the gamma function are Euler's reflection formula
\begin{align}\label{reflection}
  \Gamma(1-z)\Gamma(z)=\frac{\pi}{\sin (\pi z)},   \quad z \notin \mathbb{Z},
\end{align}
and the duplication formula
\begin{align}\label{dupl}
  \Gamma(z) \Gamma\left(z + \frac{1}{2} \right) = 2^{1-2z} \sqrt{\pi} \Gamma(2z)
\end{align}
discovered by Legendre (see \cite{HTF}, Chapter I for more on this topic).
Let $(a)_m = \frac{\Gamma(a+m)}{\Gamma(a)}$, that is,
$(a)_0 = 1$, $(a)_m = a(a + 1)\cdots (a + m - 1)$ for $m = 1, 2, \ldots$. The notation $(a)_m$ is called the Pochhammer symbol. The classical Euler-Gauss hypergeometric function is defined by
\begin{align*}
 F(a,b;c;z) =\sum_{n=0}^{\infty} \frac{(a)_n (b)_n}{(c)_n } \frac{z^n}{n!}.
\end{align*}
The series $F(a,b;c;z)$ converges when $|z| < 1$ and diverges when $|z| > 1$. For the readers's convenience, we list the properties of the function $F(a,b;c;z)$ that will be important for this paper
\begin{align}
  &F(a,b;c;1)=\frac{\Gamma(c) \Gamma(c-a-b)}{\Gamma(c-a) \Gamma(c-b)}, \quad \Re(c-a-b)>0.& \label{hyp1} \\
  &F(a,b;c;x)=(1-x)^{c-a-b} F(c-a,c-b;c;x). \label{hyp2}\\
  &F(a,b;c;x)=\frac{\Gamma(c)}{\Gamma(b) \Gamma(c-b)} \int_{0}^{1} t^{b-1} (1-t)^{c-b-1} (1-tx)^{-a} \, dt,& \label{hyp3}\\
  & \hskip26mm \Re(c)>\Re(b)>0; \, |\arg(1-x)|< \pi; \, x\neq 1.& \nonumber \\
  &\frac{d^k}{dx^k} F(a,b;c;x)=\frac{(a)_k (b)_k}{(c)_k} F(a+k,b+k;c+k;x), \quad k \in \mathbb{N}.& \label{hyp4}
\end{align}
We refer to \cite{HTF}, Chapter II for more properties of this function.
\subsection{Essential lemmas}
\begin{lem}\label{torus}
  For $a \in \mathbb{R}$, we have (see \cite{Liu})
  \begin{align*}
    \int_{\mathbb{T}} \frac{d\sigma(\zeta)}{|1- \langle z,\zeta \rangle |^{2a}}=F(a,a;1;|z|^2).
  \end{align*}
\end{lem}
It follows from the above lemma and the formula $\sum\limits_{n=0}^{\infty}x^n=\frac{1}{1-x}$ that
\begin{align}
    \int_{\mathbb{T}} \frac{d\sigma(\zeta)}{|1- \langle z,\zeta \rangle |^2}=\frac{1}{1-|z|^2}.
  \end{align}
\begin{lem}\label{fre}
  Let $c >0$ and $t >-1$. We have (see \cite{Liu})
    \begin{align*}
      \sup_{z \in \mathbb{D}} \left\{ (1-|z|^2)^c \int_{\mathbb{D}} \frac{(1-|w|^2)^t  d\nu(w)}{|1- \langle z,w \rangle|^{2+t+c}} \right\}=\frac{\Gamma(t+1) \Gamma(c)}{\Gamma^2(\frac{2+t+c}{2})}.
    \end{align*}
\end{lem}
\begin{lem}\label{mlem}
  Let $h((z_1,z_2))=(|z_2|^2 - |z_1|^2)(1-|z_2|^2)$. For $1>t>1/2$, we have
    \begin{align*}
      \sup_{(z_1,z_2) \in \mathcal{H}} \left\{ h((z_1,z_2))^t \int_{\mathcal{H}} \frac{|z_2 \overline{w}_2|  h((w_1,w_2))^{-t}  d\mu(w_1,w_2) }{|1-z_2 \overline{w}_2|^2 |z_2 \overline{w}_2 - z_1 \overline{w}_1|^2  }  \right\} = \Gamma^2(1-t)\Gamma^2(t).
    \end{align*}
\end{lem}
\begin{proof}
  Fix $1>t>1/2$ and denote $$C(t):= \sup_{(z_1,z_2) \in \mathcal{H}} \left\{ h((z_1,z_2))^t \int_{\mathcal{H}} \frac{|z_2 \overline{w}_2|  h((w_1,w_2))^{-t}  d\mu(w_1,w_2) }{|1-z_2 \overline{w}_2|^2 |z_2 \overline{w}_2 - z_1 \overline{w}_1|^2 }  \right\}.$$ Then the $C(t)$  equals
  \begin{align*}
    \sup_{(z_1,z_2) \in \mathcal{H}} \left\{ (1-|z_2|^2)^t \int_{\mathbb{D}^{*}} \frac{|z_2 \overline{w}_2| (1-|w_2|^2)^{-t} }{\pi^2 |1-z_2 \overline{w}_2|^2} \left[ \int_W A\, dV(w_1) \right]\, dV(w_2) \right\},
  \end{align*}
where $\mathbb{D}^{*}:=\{w_2 \colon 0 < |w_2| < 1\}$, $W:=\{w_1 \colon |w_1| < |w_2|\}$ and
$$A:=\frac{(|z_2|^2 - |z_1|^2)^t}{|z_2|^2 |w_2|^{2+2t}} \left( 1 - \left| \frac{w_1}{w_2} \right|^2  \right)^{-t} \left| 1- \frac{z_1 \overline{w}_1}{z_2 \overline{w}_2} \right|^{-2}.$$
Now we focus on the integral in brackets. Making the substitution $u=\frac{w_1}{w_2}$ we have
\begin{align*}
  \int_W A\, dV(w_1) = \int_{\mathbb{D}}  \frac{(|z_2|^2 - |z_1|^2)^t}{|z_2|^2 |w_2|^{2t}} \left( 1 - \left| u \right|^2  \right)^{-t} \left| 1- \frac{z_1}{z_2} \overline{u} \right|^{-2} \, dV(u).
\end{align*}
For fixed $z_2 \in \mathbb{D}^{*}$, by Lemma \ref{fre}, we have
\begin{align*}
  \sup_{|z_1|<|z_2|} \left\{ \int_W A\, dV(w_1) \right\} = \frac{\pi \Gamma(1-t)\Gamma(t)}{|w_2|^{2t} |z_2|^{2-2t}}.
\end{align*}
Therefore
  \begin{align*}
    C(t) = \sup_{z_2 \in \mathbb{D}^{*}} \left\{ \Gamma(1-t)\Gamma(t)  \int_{\mathbb{D}^{*}} \frac{(1-|z_2|^2)^t |z_2|^{2t-1} |w_2|^{1-2t}  }{\pi (1-|w_2|^2)^{t} |1-z_2 \overline{w}_2|^2  } \, dV(w_2) \right\}.
  \end{align*}
For $z_2 \in \mathbb{D}^{*}$, denote $$ I(z_2)= \int_{\mathbb{D}^{*}} \frac{(1-|z_2|^2)^t |z_2|^{2t-1} |w_2|^{1-2t}  }{\pi (1-|w_2|^2)^{t} |1-z_2 \overline{w}_2|^2  } \, dV(w_2). $$
Introducing polar coordinate in $w_2=r e^{i\zeta}$ variable we have
\begin{align*}
  I(z_2)=2(1-|z_2|^2)^t |z_2|^{2t-1}  \int_{0}^{1} \left[\int_{\mathbb{T}} \frac{d\sigma(\zeta)}{|1-\langle rz_2,\zeta \rangle|^2}  \right] \frac{r^{2-2t} \, dr} {(1-r^2)^{t}}
\end{align*}
Next, by Lemma \ref{torus}
 \begin{align*}
  I(z_2)=2(1-|z_2|^2)^t |z_2|^{2t-1}  \int_{0}^{1}  \frac{r^{2-2t} (1-r^2)^{-t}}{1-r^2|z_2|^2} \,dr
\end{align*}
Hence, by (\ref{hyp3}) we obtain
\begin{align*}
  I(z_2)=(1-|z_2|^2)^t |z_2|^{2t-1}  \frac{2^{2t-1} \sqrt{\pi} \Gamma(2-2t)}{\Gamma\left(\frac{5}{2}-2t \right)} F\left(1,\frac{3}{2}-t; \frac{5}{2} -2t; |z_2|^2 \right).
\end{align*}
Finally, by (\ref{hyp2})
 \begin{align*}
  I(z_2)=   \frac{2^{2t-1} \sqrt{\pi} \Gamma(2-2t)}{\Gamma\left(\frac{5}{2}-2t \right)} |z_2|^{2t-1} F\left(\frac{3}{2}-2t,1-t; \frac{5}{2} -2t; |z_2|^2 \right).
\end{align*}
If $3/2-2t>0 \Leftrightarrow 3/4>t$, then $|z_2|^{2t-1} F\left(\frac{3}{2}-2t,1-t; \frac{5}{2} -2t; |z_2|^2 \right)$ is the increasing function of $|z_2|$ ($|z_2| \in [0,1)$), since its Taylor coefficients are all positive. Therefore, by (\ref{hyp1})
  \begin{align*}
    C(t) =  \Gamma(1-t)\Gamma(t)  \frac{2^{2t-1} \sqrt{\pi} \Gamma(2-2t)}{\Gamma\left(\frac{5}{2}-2t \right)}  \frac{\Gamma(t) \Gamma(2-2t)}{\Gamma\left(\frac{3}{2}-t \right)}.
  \end{align*}
Using the duplication formula (\ref{dupl}), we get
 \begin{align*}
    C(t) = \Gamma^2(1-t)\Gamma^2(t).
  \end{align*}
In the case when $3/2-2t \leq 0$ we consider the function
$$f(\lambda):=\lambda^{t-1/2} F\left(\frac{3}{2}-2t,1-t; \frac{5}{2} -2t; \lambda \right), \quad \lambda \in [0,1].$$
Applying differentiation formula of the hypergeometric function (\ref{hyp4}), we have
\begin{align*}
  f'(\lambda)=\lambda^{t-3/2} g(\lambda),
\end{align*}
where
\begin{align*}
   g(\lambda)=&\left(t-\frac{1}{2}\right) F\left(\frac{3}{2}-2t,1-t; \frac{5}{2} -2t; \lambda \right) \\ &+ \lambda F\left(\frac{5}{2}-2t,2-t; \frac{7}{2} -2t; \lambda \right).
\end{align*}
Since $g(0)=t-1/2>0$ and  $g'(\lambda) >0 $ for $\lambda \in (0,1)$, we obtain that the function $f$ is an increasing function on an interval $[0,1]$. Therefore the conclusion about a constant $C(t)$ is the same as in the case when $3/2-2t>0$ which completes the proof.
\end{proof}
\begin{lem}\label{cal}\cite{Liu}
Let $a, b, c \in \mathbb{R}$ and $t >-1$. The identity
\begin{align*}
  \int_{\mathbb{D}}& \frac{(1-|\xi|^2)^t d\nu(\xi)}{(1- \langle z,\xi \rangle)^a (1- \langle w,\xi \rangle)^b (1- \langle \xi,w \rangle)^c} \\ &=\frac{\Gamma(1+t)}{\Gamma(2+t)} \sum_{j=0}^{\infty} \frac{(a)_j (c)_j}{(2+t)_j j!} F\left(b,c+j;2+t+j;|w|^2\right) \langle z,w\rangle^j
  \end{align*}
holds for any $z,w \in \mathbb{D}$.
\end{lem}
\begin{lem}\label{rozk}\cite{Liu}
Let $1<p<\infty$ and
\begin{align*}
  &\Psi_{\xi}(z):=\Gamma(2/p)\Gamma(2/q) \sum_{k=0}^{\infty} \epsilon_k \langle z, \xi \rangle^k, \\
  &\Upsilon_{\xi}(z):= \sum_{k=0}^{\infty} a_k(\xi) \langle z, \xi \rangle^k,
\end{align*}
where
\begin{align*}
  & \epsilon_k:= \frac{(2/p)_k}{k!} \left( \frac{\Gamma(k+2) \Gamma(k+1)}{ \Gamma(k+1+2/q) \Gamma(k+2/p) } -1 \right), \\
  & a_k(\xi):= \frac{(1)_k}{k!} \left( F(2/p-1,k+1;k+2;|\xi|^2) - \frac{\Gamma(2/q) \Gamma(k+2) }{\Gamma(k+1+2/q) } \right).
\end{align*}
Then we have
\begin{align*}
  &\sup_{\xi \in \mathbb{D}}\left\{\int_{\mathbb{D}} |\Psi_{\xi}(z)|^p \, d\nu(z) \right\} < \infty, \\
  &\sup_{\xi \in \mathbb{D}} \left\{ \int_{\mathbb{D}} |\Upsilon_{\xi}(z)|^p \, d\nu(z) \right\} < \infty.
\end{align*}
\end{lem}
\section{Main result}
Our  main result reads as follows.
\begin{thm}[Main Theorem]
  For $4/3<p<4$, we have
   \begin{align}\label{main}
    \Gamma^2\left(\frac{2}{p}\right)\Gamma^2\left(\frac{2}{q}\right) \leq \|\mathbf{P}\|_p \leq  \Gamma^2\left(1-\frac{2}{p}\right) \Gamma^2\left(\frac{2}{p}\right),
   \end{align}
where $q:=\frac{p}{p-1}$ is the conjugate exponent of $p$.
\end{thm}
\begin{proof}
  First we prove the upper estimate in (\ref{main}). To do so, we use the well known Schur's test (see, for instance, \cite{Zhu} Theorem 3.6).
\begin{lem}
Suppose that $(X, \rho)$ is a $\sigma$-finite measure space and $K(x, y)$ is a nonnegative measurable function on $X \times X$ and $T$ the associated integral operator
$$Tf(x)=\int_X K(x,y) f(y) \, d\rho(y).$$
Let $1 < p < \infty$ and $1/p +1/q=1$. If there exist a positive constant $C$ and a positive measurable function $h$ on $X$ such that
$$\int_X K(x,y) h(y)^q \, d\rho(y) \leq C h(x)^q$$
for almost every $x$ in $X$ and
$$\int_X K(x,y) h(x)^p \, d\rho(x) \leq C h(y)^p$$
for almost every $y$ in $X$, then $T$ is bounded on $L^p(X, d\rho)$ with $\|T\| \leq C$.
\end{lem}
We only need to consider the case when $4>p \geq 2$, and the case when $4/3 <p <2$ then follows from the duality. If we put
\begin{align*}
  &K((z_1,z_2),(w_1,w_2))= \frac{z_2 \overline{w}_2}{(1-z_2 \overline{w}_2)^2 (z_2 \overline{w}_2 - z_1 \overline{w}_1)^2},\\
  &h((z_1,z_2))=\left[(|z_2|^2 - |z_1|^2)(1-|z_2|^2)\right]^{-\frac{2}{pq}}, \\
  &C(p)=\sup_{(z_1,z_2) \in \mathcal{H}} \left\{ h((z_1,z_2))^{2/p} \int_{\mathcal{H}} \frac{|z_2 \overline{w}_2|  h((w_1,w_2))^{-2/p}  d\mu(w_1,w_2) }{|1-z_2 \overline{w}_2|^2 |z_2 \overline{w}_2 - z_1 \overline{w}_1|^2 }  \right\},
\end{align*}
where $q$ is the conjugate exponent of $p$, it is clear that
\begin{align*}
   \int_{ \mathcal{H}} K(z,w) h(w)^q \, d\mu(w) \leq C(p) h(z)^q\\
  \int_{ \mathcal{H}} K(z,w) h(z)^p \, d\mu(z) \leq C(q) h(w)^p
\end{align*}
for almost every $z \in \mathcal{H}$ and $w \in \mathcal{H}$, respectively. From Lemma \ref{mlem}
$$C(p)=\Gamma^2\left(1-\frac{2}{p}\right) \Gamma^2\left(\frac{2}{p}\right)=C(q).$$
Hence, an application of Schur's test gives the desired upper estimate. To prove the lower estimate, we define, for
$(z_1,z_2), (\xi_1,\xi_2) \in \mathcal{H}$
\begin{align*}
  f_{(\xi_1,\xi_2)} ((z_1,z_2)) :=\frac{(1-\xi_2 \overline{z}_2)^{1-\frac{2}{p}}}{z_2(1-z_2 \overline{\xi}_2)} \frac{ \left(1-\frac{\xi_1}{\xi_2} \frac{\overline{z}_1}{\overline{z}_2} \right)^{1-\frac{2}{p}} }{ 1-\frac{z_1}{z_2} \frac{\overline{\xi}_1}{\overline{\xi}_2}}.
\end{align*}
We show that
\begin{align}\label{f}
  \left\|f_{(\xi_1,\xi_2)}\right\|^p_p \approx \log \frac{1}{1-|\xi_2|^2} \log \frac{1}{1-\left| \xi_1/ \xi_2 \right|^2}, \quad \text{as} \quad |\xi_2|, \left| \xi_1/ \xi_2 \right| \rightarrow 1^{-}.
\end{align}
By the definition we have
\begin{align*}
  \left\|f_{(\xi_1,\xi_2)}\right\|^p_p &= \int_{\mathcal{H}} \frac{|1-\xi_2 \overline{z}_2|^{p-2}}{|z_2|^p \left|1-z_2 \overline{\xi}_2\right|^p} \frac{ \left|1-\frac{\xi_1}{\xi_2} \frac{\overline{z}_1}{\overline{z}_2} \right|^{p-2} }{ \left|1-\frac{z_1}{z_2} \frac{\overline{\xi}_1}{\overline{\xi}_2} \right|^p} \, d\mu(z_1,z_2) \\ &= \int_{\mathcal{H}} \frac{1}{|z_2|^p \left|1-z_2 \overline{\xi}_2\right|^2} \frac{ 1 }{ \left|1-\frac{z_1}{z_2} \frac{\overline{\xi}_1}{\overline{\xi}_2} \right|^2} \, d\mu(z_1,z_2) \\ &= \int_{\mathbb{D}^{*}} \frac{|z_2|^{-p}}{\left|1-z_2 \overline{\xi}_2\right|^2} \left[ \int_{|z_1|<|z_2|}  \frac{ 1 }{ \left|1-\frac{z_1}{z_2} \frac{\overline{\xi}_1}{\overline{\xi}_2} \right|^2} \, d\nu(z_1) \right] \, d\nu(z_2).
\end{align*}
Making the substitution $u=\frac{z_1}{z_2}$, we have
\begin{align*}
  \left\|f_{(\xi_1,\xi_2)}\right\|^p_p = \int_{\mathbb{D}^{*}} \frac{|z_2|^{2-p}}{\left|1-z_2 \overline{\xi}_2\right|^2} \left[ \int_{\mathbb{D}}  \frac{ 1 }{ \left|1- \overline{\xi}_1/ \overline{\xi}_2 u \right|^2} \, d\nu(u) \right] \, d\nu(z_2).
\end{align*}
Now (\ref{f}) follows from the well known Forelli-Rudin estimate (see \cite{R}, Proposition 1.4.10). \newline
A similar calculations show that
\begin{align*}
  \mathbf{P}f_{(\xi_1,\xi_2)} ((z_1,z_2))= \int_{ \mathbb{D}^{*}} \frac{ (1-\xi_2 \overline{z}_2)^{1-\frac{2}{p}}  \, d\nu(w_2)}{z_2(1-z_2 \overline{w}_2)^2 (1-w_2 \overline{\xi}_2)} \int_{\mathbb{D}} \frac{ \left(1-\frac{\xi_1}{\xi_2} \overline{u} \right)^{1-\frac{2}{p}} \, d\nu(u)}{ \left( 1-\frac{z_1}{z_2} \overline{u}\right)^2 \left(1-u  \frac{\overline{\xi}_1}{\overline{\xi}_2} \right)}.
\end{align*}
Using twice Lemma \ref{cal}, we get
\begin{align*}
  \mathbf{P}f_{(\xi_1,\xi_2)} ((z_1,z_2))=&\frac{1}{z_2} \sum_{k=0}^{\infty} \frac{(1)_k}{k!} F\left(2/p-1,1+k;2+k;|\xi_2|^2\right) \langle z_2,\xi_2 \rangle^k \\ &\cdot \sum_{l=0}^{\infty} \frac{(1)_l}{l!} F\left(2/p-1,1+l;2+l;|\xi_2/ \xi_2|^2\right) \left\langle \frac{z_1}{z_2}, \frac{\xi_1}{\xi_2} \right\rangle^k.
\end{align*}
It is easy to check that
\begin{align*}
  \mathbf{P}f_{(\xi_1,\xi_2)} ((z_1,z_2))=&\frac{1}{z_2}\left( \Phi_{\xi_2}(z_2) + \Psi_{\xi_2}(z_2) + \Upsilon_{\xi_2}(z_2)  \right) \\ & \cdot \left(\Phi_{\xi_1 / \xi_2}(z_1/z_2) + \Psi_{\xi_1 / \xi_2}(z_1/ z_2) + \Upsilon_{\xi_2/ \xi_2}(z_1/z_2) \right),
\end{align*}
where $\Phi_{\xi}(z)=\Gamma(2/p)\Gamma(2/q)\left(1- z \overline{\xi} \right)^{-2/p}$.
Hence
\begin{align*}
  \|\mathbf{P}\|_p &\geq \limsup_{|\xi_2|,|\xi_1 / \xi_2|\rightarrow 1^{-}} \frac{\| \mathbf{P}f_{(\xi_1,\xi_2)}\|_p}{\|f_{(\xi_1,\xi_2)}\|_p} \\& \geq \limsup_{|\xi_2|,|\xi_1 / \xi_2|\rightarrow 1^{-}} \left(\frac{ \left\|\frac{1}{z_2} \Phi_{\xi_2}(z_2) \Phi_{\xi_1 / \xi_2}(z_1/z_2)\right\|_p } {\|f_{(\xi_1,\xi_2)}\|_p} - \frac{R}{\|f_{(\xi_1,\xi_2)}\|_p} \right),
\end{align*}
where
\begin{align*}
R=&\left\|\frac{1}{z_2} \Phi_{\xi_2}(z_2) \Psi_{\xi_1 / \xi_2}(z_1/ z_2)\right\|_p + \left\|\frac{1}{z_2} \Phi_{\xi_2}(z_2) \Upsilon_{\xi_1 / \xi_2}(z_1/ z_2)\right\|_p\\    &+ \left\|\frac{1}{z_2} \Psi_{\xi_2}(z_2) \Phi_{\xi_1 / \xi_2}(z_1/ z_2)\right\|_p + \left\|\frac{1}{z_2} \Psi_{\xi_2}(z_2) \Upsilon_{\xi_1 / \xi_2}(z_1/ z_2)\right\|_p   \\&+ \left\|\frac{1}{z_2} \Upsilon_{\xi_2}(z_2) \Phi_{\xi_1 / \xi_2}(z_1/ z_2)\right\|_p + \left\|\frac{1}{z_2} \Upsilon_{\xi_2}(z_2) \Psi_{\xi_1 / \xi_2}(z_1/ z_2)\right\|_p
\\  &+ \left\|\frac{1}{z_2} \Psi_{\xi_2}(z_2) \Psi_{\xi_1 / \xi_2}(z_1/ z_2)\right\|_p + \left\|\frac{1}{z_2} \Upsilon_{\xi_2}(z_2) \Upsilon_{\xi_1 / \xi_2}(z_1/ z_2)\right\|_p.
\end{align*}
It is obvious that
\begin{align*}
   \left\|\frac{1}{z_2} \Phi_{\xi_2}(z_2) \Phi_{\xi_1 / \xi_2}(z_1/z_2)\right\|_p=\Gamma^2(2/p)\Gamma^2(2/q) \|f_{(\xi_1,\xi_2)}\|_p.
\end{align*}
The following lemma completes the proof of the main result
\begin{lem} Define $R$ as above. Then for $4/3<p<4$
\begin{align}
\limsup_{|\xi_2|,|\xi_1 / \xi_2|\rightarrow 1^{-}} \frac{R}{\|f_{(\xi_1,\xi_2)}\|_p}=0.
\end{align}
\end{lem}
But to avoid disrupting the flow of the paper with several pages of computations, we postpone its proof to the next section.
\end{proof}
\section{Proof of Lemma 3.2}
It is enough to show that each term of $R$ divided by $\|f_{(\xi_1,\xi_2)}\|_p$ goes to zero as $|\xi_2|,|\xi_1 / \xi_2|\rightarrow 1^{-}$. We start with
\begin{align*}
  \left\|\frac{1}{z_2} \Phi_{\xi_2}(z_2) \Psi_{\xi_1 / \xi_2}(z_1/ z_2)\right\|^p_p=\int_{\mathcal{H}} \left|\frac{1}{z_2} \Phi_{\xi_2}(z_2) \Psi_{\xi_1 / \xi_2}(z_1/ z_2)\right|^p \, d\mu(z_1,z_2).
\end{align*}
Making the substitution $u=z_1/z_2$, we have
\begin{align*}
  \left\|\frac{1}{z_2} \Phi_{\xi_2}(z_2) \Psi_{\xi_1 / \xi_2}(z_1/ z_2)\right\|^p_p= \int_{ \mathbb{D}^{*}} \frac{|\Phi_{\xi_2}(z_2)|^p}{|z_2|^{p-2}} \left\{ \int_{\mathbb{D}} |\Psi_{\xi_1 / \xi_2}(u)|^p \,d\nu(u) \right\} \, d\nu(z_2).
\end{align*}
Lemma \ref{rozk}  yields that there exists a constant $C>0$ such that
\begin{align*}
  \left\|\frac{1}{z_2} \Phi_{\xi_2}(z_2) \Psi_{\xi_1 / \xi_2}(z_1/ z_2)\right\|^p_p \leq  C \int_{ \mathbb{D}^{*}} \frac{|\Phi_{\xi_2}(z_2)|^p}{|z_2|^{p-2}} \, d\nu(z_2).
\end{align*}
Since $\int_{ \mathbb{D}^{*}} |z_2|^{2-p} \, d\nu(z_2)$ is finite (when $p<4$) and $\Phi_{\xi_2}(0)=\Gamma(2/p)\Gamma(2/q)$ we can write
\begin{align*}
  \left\|\frac{1}{z_2} \Phi_{\xi_2}(z_2) \Psi_{\xi_1 / \xi_2}(z_1/ z_2)\right\|^p_p \leq  C_1 \int_{ \mathbb{D}^{*}} |\Phi_{\xi_2}(z_2)|^p \, d\nu(z_2),
\end{align*}
for some constant $C_1$. Therefore by the Forelli-Rudin estimate
\begin{align*}
  \left\|\frac{1}{z_2} \Phi_{\xi_2}(z_2) \Psi_{\xi_1 / \xi_2}(z_1/ z_2)\right\|^p_p \approx \log \frac{1}{1-|\xi_2|^2}, \quad \text{as} \quad |\xi_2| \rightarrow 1^{-}.
\end{align*}
Hence, by (\ref{f})
\begin{align*}
   \limsup_{|\xi_2|,|\xi_1 / \xi_2|\rightarrow 1^{-}} \frac{\left\|\frac{1}{z_2} \Phi_{\xi_2}(z_2) \Psi_{\xi_1 / \xi_2}(z_1/ z_2)\right\|_p}{\|f_{(\xi_1,\xi_2)}\|_p}=0.
\end{align*}
The limit
\begin{align*}
   \limsup_{|\xi_2|,|\xi_1 / \xi_2|\rightarrow 1^{-}} \frac{\left\|\frac{1}{z_2} \Phi_{\xi_2}(z_2) \Upsilon_{\xi_1 / \xi_2}(z_1/ z_2)\right\|_p}{\|f_{(\xi_1,\xi_2)}\|_p}=0.
\end{align*}
can be obtained by the similar method and we omit the details. \newline Let us now consider
\begin{align*}
  \left\|\frac{1}{z_2} \Psi_{\xi_2}(z_2) \Phi_{\xi_1 / \xi_2}(z_1/ z_2)\right\|^p_p=\int_{\mathcal{H}} \left| \frac{1}{z_2} \Psi_{\xi_2}(z_2) \Phi_{\xi_1 / \xi_2}(z_1/ z_2) \right|^p \, d\mu(z_1,z_2).
\end{align*}
As before we make a substitution $u=z_1/z_2$
\begin{align*}
  \left\|\frac{1}{z_2} \Psi_{\xi_2}(z_2) \Phi_{\xi_1 / \xi_2}(z_1/ z_2)\right\|^p_p=\int_{ \mathbb{D}^{*}} \frac{|\Psi_{\xi_2}(z_2)|^p}{|z_2|^{p-2}} \left\{ \int_{\mathbb{D}} |\Phi_{\xi_1 / \xi_2}(u)|^p \, d\nu(u) \right\} \, d\nu(z_2).
\end{align*}
From the Forelli-Rudin estimate there exists a constant $C>0$ such that
\begin{align*}
  \left\|\frac{1}{z_2} \Psi_{\xi_2}(z_2) \Phi_{\xi_1 / \xi_2}(z_1/ z_2)\right\|^p_p \leq  C \log \frac{1}{1-\left| \xi_1/ \xi_2 \right|^2} \int_{ \mathbb{D}^{*}} \frac{|\Psi_{\xi_2}(z_2)|^p}{|z_2|^{p-2}}  \, d\nu(z_2).
\end{align*}
Since $\int_{ \mathbb{D}^{*}} |z_2|^{2-p} \, d\nu(z_2)$ is finite (when $p<4$) and $$\Psi_{\xi_2}(0)=\Gamma(2/p)\Gamma(2/q)\left(\frac{1}{\Gamma(1+2/q)\Gamma(2/p)}-1\right),$$
then
\begin{align*}
  \left\|\frac{1}{z_2} \Psi_{\xi_2}(z_2) \Phi_{\xi_1 / \xi_2}(z_1/ z_2)\right\|^p_p \leq  C_1 \log \frac{1}{1-\left| \xi_1/ \xi_2 \right|^2}.
\end{align*}
Therefore, by (\ref{f})
\begin{align*}
   \limsup_{|\xi_2|,|\xi_1 / \xi_2|\rightarrow 1^{-}} \frac{\left\|\frac{1}{z_2} \Psi_{\xi_2}(z_2) \Phi_{\xi_1 / \xi_2}(z_1/ z_2)\right\|_p}{\|f_{(\xi_1,\xi_2)}\|_p}=0.
\end{align*}
A similar calculation reveals that
\begin{align*}
  \left\|\frac{1}{z_2} \Upsilon_{\xi_2}(z_2) \Phi_{\xi_1 / \xi_2}(z_1/ z_2)\right\|^p_p \leq  C \log \frac{1}{1-\left| \xi_1/ \xi_2 \right|^2} \int_{ \mathbb{D}^{*}} \frac{|\Upsilon_{\xi_2}(z_2)|^p}{|z_2|^{p-2}}  \, d\nu(z_2),
\end{align*}
for some constant $C$. Since
\begin{align*}
  &\Upsilon_{0}(0)= 1 - \frac{\Gamma(2/q) \Gamma(2) }{\Gamma(1+2/q)}, \\
  &\limsup_{|\xi_2| \rightarrow 1^{-}}  \Upsilon_{\xi_2}(0)=\frac{\Gamma(2-2/p) \Gamma(2) }{\Gamma(3-2/p)} - \frac{\Gamma(2/q) \Gamma(2) }{\Gamma(1+2/q)},
\end{align*}
we conclude that
\begin{align*}
   \limsup_{|\xi_2|,|\xi_1 / \xi_2|\rightarrow 1^{-}} \frac{\left\|\frac{1}{z_2} \Upsilon_{\xi_2}(z_2) \Phi_{\xi_1 / \xi_2}(z_1/ z_2)\right\|_p}{\|f_{(\xi_1,\xi_2)}\|_p}=0.
\end{align*}
Next we investigate
\begin{align*}
   \left\|\frac{1}{z_2} \Upsilon_{\xi_2}(z_2) \Upsilon_{\xi_1 / \xi_2}(z_1/ z_2)\right\|^p_p \quad \text{and} \quad
  \left\|\frac{1}{z_2} \Upsilon_{\xi_2}(z_2) \Psi_{\xi_1 / \xi_2}(z_1/ z_2)\right\|^p_p.
\end{align*}
After a change of variables
\begin{align*}
   \left\|\frac{\Upsilon_{\xi_2}(z_2) \Upsilon_{\xi_1 / \xi_2}(z_1/ z_2)}{z_2} \right\|^p_p=\int_{ \mathbb{D}^{*}} \frac{|\Upsilon_{\xi_2}(z_2)|^p}{|z_2|^{p-2}} \left\{ \int_{\mathbb{D}} |\Upsilon_{\xi_1 / \xi_2}(u)|^p \, d\nu(u) \right\} \, d\nu(z_2),\\
  \left\|\frac{\Upsilon_{\xi_2}(z_2) \Psi_{\xi_1 / \xi_2}(z_1/ z_2)}{z_2} \right\|^p_p=\int_{ \mathbb{D}^{*}} \frac{|\Upsilon_{\xi_2}(z_2)|^p}{|z_2|^{p-2}} \left\{ \int_{\mathbb{D}} |\Psi_{\xi_1 / \xi_2}(u)|^p \, d\nu(u) \right\} \, d\nu(z_2).
\end{align*}
Hence, (using  Lemma \ref{rozk})
\begin{align*}
   \left\|\frac{1}{z_2} \Upsilon_{\xi_2}(z_2) \Upsilon_{\xi_1 / \xi_2}(z_1/ z_2)\right\|^p_p \leq C \int_{ \mathbb{D}^{*}} \frac{|\Upsilon_{\xi_2}(z_2)|^p}{|z_2|^{p-2}} \, d\nu(z_2),\\
  \left\|\frac{1}{z_2} \Upsilon_{\xi_2}(z_2) \Psi_{\xi_1 / \xi_2}(z_1/ z_2)\right\|^p_p \leq D \int_{ \mathbb{D}^{*}} \frac{|\Upsilon_{\xi_2}(z_2)|^p}{|z_2|^{p-2}}  \, d\nu(z_2).
\end{align*}
for some constants $C,D$. Since the function $\Upsilon_{\xi_2}(0)$ (as a function of $\xi$) is bounded on $\mathbb{D}$ and $\int_{ \mathbb{D}^{*}} |z_2|^{2-p} \, d\nu(z_2)$ is finite, it is easy to see that
 \begin{align*}
   \left\|\frac{1}{z_2} \Upsilon_{\xi_2}(z_2) \Upsilon_{\xi_1 / \xi_2}(z_1/ z_2)\right\|^p_p \leq C_1 \int_{ \mathbb{D}^{*}} |\Upsilon_{\xi_2}(z_2)|^p \, d\nu(z_2),\\
  \left\|\frac{1}{z_2} \Upsilon_{\xi_2}(z_2) \Psi_{\xi_1 / \xi_2}(z_1/ z_2)\right\|^p_p \leq D_1 \int_{ \mathbb{D}^{*}} |\Upsilon_{\xi_2}(z_2)|^p  \, d\nu(z_2).
\end{align*}
Apply Lemma \ref{rozk} again, we have
 \begin{align*}
  \sup_{(\xi_1,\xi_2) \in \mathcal{H}} \left\|\frac{1}{z_2} \Upsilon_{\xi_2}(z_2) \Upsilon_{\xi_1 / \xi_2}(z_1/ z_2)\right\|_p \leq \infty,\\
  \sup_{(\xi_1,\xi_2) \in \mathcal{H}} \left\|\frac{1}{z_2} \Upsilon_{\xi_2}(z_2) \Psi_{\xi_1 / \xi_2}(z_1/ z_2)\right\|_p \leq \infty.
\end{align*}
Thus
\begin{align*}
   \limsup_{|\xi_2|,|\xi_1 / \xi_2|\rightarrow 1^{-}} \frac{\left\|\frac{1}{z_2} \Upsilon_{\xi_2}(z_2) \Upsilon_{\xi_1 / \xi_2}(z_1/ z_2)\right\|_p}{\|f_{(\xi_1,\xi_2)}\|_p}=0, \\ \limsup_{|\xi_2|,|\xi_1 / \xi_2|\rightarrow 1^{-}} \frac{\left\|\frac{1}{z_2} \Upsilon_{\xi_2}(z_2) \Psi_{\xi_1 / \xi_2}(z_1/ z_2)\right\|_p}{\|f_{(\xi_1,\xi_2)}\|_p}=0.
\end{align*}
This is what we wanted to establish. A similar calculations show that
\begin{align*}
   \limsup_{|\xi_2|,|\xi_1 / \xi_2|\rightarrow 1^{-}} \frac{\left\|\frac{1}{z_2} \Psi_{\xi_2}(z_2) \Psi_{\xi_1 / \xi_2}(z_1/ z_2)\right\|_p}{\|f_{(\xi_1,\xi_2)}\|_p}=0, \\ \limsup_{|\xi_2|,|\xi_1 / \xi_2|\rightarrow 1^{-}} \frac{\left\|\frac{1}{z_2} \Psi_{\xi_2}(z_2) \Upsilon_{\xi_1 / \xi_2}(z_1/ z_2)\right\|_p}{\|f_{(\xi_1,\xi_2)}\|_p}=0.
\end{align*}
That completes the proof.
\bibliographystyle{amsplain}

\noindent Tomasz Beberok\\
Department of Applied Mathematics,\\
University of Agriculture in Krakow,\\
ul. Balicka 253c, 30-198 Kraków, Poland\\
email: tbeberok@ar.krakow.pl

% ------------------------------------------------------------------------
\end{document}